\documentclass[reqno,12pt,letterpaper]{amsart}

\usepackage{amsmath,amssymb,amsthm,graphicx,mathrsfs,url,bm,subfigure,accents,paralist,xargs, slashed,enumitem,mathtools,array}
\usepackage[usenames,dvipsnames]{color}
\usepackage{tikz}
\usetikzlibrary{decorations.markings}
\usepackage[colorlinks=true,linkcolor=Red,citecolor=Green]{hyperref}

\def\?[#1]{\textbf{[#1]}\marginpar{\Large{\textbf{??}}}}

\newlist{inlineroman}{enumerate*}{1}
\setlist[inlineroman]{itemjoin*={{, and }},afterlabel=~,label=\roman*.}

\newcommand{\inlineitem}[1][]{%
	\ifnum\enit@type=\tw@
	{\descriptionlabel{#1}}
	\hspace{\labelsep}
	\else
	\ifnum\enit@type=\z@
	\refstepcounter{\@listctr}\fi
	\quad\@itemlabel\hspace{\labelsep}
	\fi}

\DeclareFontEncoding{FMS}{}{}
\DeclareFontSubstitution{FMS}{futm}{m}{n}
\DeclareFontEncoding{FMX}{}{}
\DeclareFontSubstitution{FMX}{futm}{m}{n}
\DeclareSymbolFont{fouriersymbols}{FMS}{futm}{m}{n}
\DeclareSymbolFont{fourierlargesymbols}{FMX}{futm}{m}{n}
\DeclareMathDelimiter{\VERT}{\mathord}{fouriersymbols}{152}{fourierlargesymbols}{147}

\theoremstyle{plain}
\newtheorem{theo}{Theorem}
\newenvironment{thmbis}[1]
{\renewcommand{\thetheo}{\ref{#1}$'$}%
	\addtocounter{theo}{-1}%
	\begin{theo}}
	{\end{theo}}

\newtheorem{prop}{Proposition}[section]

\newtheorem{lem}[prop]{Lemma}
\newtheorem{cor}{Corollary}

\theoremstyle{remark}

\theoremstyle{definition}
\newtheorem{defi}{Definition}

\numberwithin{equation}{section}

\DeclareMathOperator{\supp}{supp}

\DeclareMathOperator{\vol}{vol}

\newcommand{\divop}{\mathrm{div}}

\newcommand{\RR}{\mathbb{R}}
\newcommand{\CC}{\mathbb{C}}
\newcommand{\bP}{\mathbf{P}}

\newcommand{\Neg}{\mathcal{N}}

\renewcommand{\Im}{\operatorname{Im}}
\renewcommand{\Re}{\operatorname{Re}}

\title[Resolvent estimates for spacetimes bounded by Killing horizons]
{Resolvent estimates for spacetimes bounded by Killing horizons}

\author{Oran Gannot}
\email{gannot@northwestern.edu}
\address{Department of Mathematics, Lunt Hall, Northwestern University,
	Evanston, CA 60208, USA}

\begin{document}

\begin{abstract}
We show that the resolvent grows at most exponentially with frequency for the wave equation on a class of stationary spacetimes which are bounded by non-degenerate Killing horizons, without any assumptions on the trapped set.  Correspondingly, there exists an exponentially small resonance-free region, and solutions of the Cauchy problem exhibit logarithmic energy decay.
\end{abstract}

	\maketitle

\section{Introduction}

\subsection{Statement of results} \label{subsect:mainresults} Let $(M,g)$ be a connected $n+1$ dimensional Lorentzian manifold of signature $(1,n)$ with connected boundary $\partial M$, satisfying the following assumptions.

\begin{enumerate} \itemsep6pt
\item $\partial M$ is a Killing horizon generated by a complete Killing vector field $T$, whose surface gravity is a positive constant $\kappa > 0$ (see Section \ref{subsect:killing} for details),

\item $M$ is stationary in the sense that there is a compact spacelike hypersurface $X$ with boundary such that each integral curve of $T$ intersects $X$ exactly once,	

\item $T$ is timelike in $M^\circ$.

\end{enumerate}
Consider a formally self-adjoint (with respect to the volume density) operator $L \in \mathrm{Diff}^2(M)$ commuting with $T$, such that $L - \Box_g \in \mathrm{Diff}^1(M)$. Thus we can write
\[
L = \Box_g + \mathcal{W}  +\mathcal{V},
\]
where $\mathcal W$ is a smooth vector field and $\mathcal V \in \mathcal{C}^\infty(M)$. In addition, assume that $\mathcal{W}$ is tangent to $\partial M$.

%
%

Identify $M = \RR_t \times X$ under the flow of $T$. Since $T$ commutes with $L$, the composition
\begin{equation} \label{eq:stationaryoperator}
\mathbf{P}(\omega) = e^{i\omega t} L e^{-i\omega t}
\end{equation}
descends to a differential operator on $X$ depending on $\omega \in \CC$. Fredholm properties of $\bP(\omega)$ were first examined in a robust fashion by Vasy \cite{vasy2013microlocal} using methods of microlocal analysis, and subsequently by Warnick \cite{warnick2015quasinormal} via physical space arguments (see also \cite{gannot2017null}).

Here we summarize a simple version of these results, which applies in any strip of fixed width near the real axis.  For $k \in \mathbb{N}$, let
\begin{equation} \label{eq:Xspace}
\mathcal{X}^k = \{u \in H^{k+1}(X): \bP(0)u \in H^k(X)\},
\end{equation}
equipped with the graph norm. Since $\bP(\omega) - \bP(0) \in \mathrm{Diff}^1(X)$, the operator $\bP(\omega)$ is bounded $\mathcal{X}^k \rightarrow H^k(X)$ for each $\omega \in \CC$.

\begin{prop}[\cite{vasy2013microlocal}, \cite{warnick2015quasinormal}] \label{theo:vasyclassical}
	The operator $\bP(\omega) : \mathcal{X}^k \rightarrow H^k(X)$ is Fredholm of index zero in the half-plane $\{\Im \omega > -\kappa(k+1/2)\}$, and is invertible for $\Im \omega > 0$ sufficiently large.
\end{prop}
The inverse $\bP(\omega)^{-1} :H^k(X) \rightarrow \mathcal{X}^k$ forms a meromorphic family of operators in $\{\Im \omega > -\kappa(k+1/2)\}$, called the resolvent family, which is independent of $k$ in a suitable sense \cite[Remark 2.9]{vasy2013microlocal}. Its complex poles in $\{\Im \omega > -\kappa(k+ 1/2)\}$ are known as resonances, and correspond to nontrivial mode solutions $v = e^{-i\omega t}u$ of the equation $\Box_g v = 0$, where $u \in \mathcal{C}^\infty(M)$ satisfies $Tu =0$. Thus mode solutions with $\Im \omega > 0$ grow exponentially in time, whereas those with $\Im \omega < 0$ exhibit exponential decay.

Given $\omega_0, C_0 >0$, define the region
\[
\Omega = \{ |\Im \omega| \leq e^{-C_0 |\Re \omega|}\} \cap \{|\omega| > \omega_0\}.
\]
These parameters are fixed in the next theorem, which is the main result of this paper.

\begin{theo}\label{theo:maintheo}
	There exist $\omega_0,C_0 > 0$ such that $\bP(\omega)$ has no resonances in $\Omega$. Furthermore, there exists $C>0$ such that if $\omega \in \Omega$, then
	\begin{equation} \label{eq:mainestimate}
	\| \bP(\omega)^{-1} f \|_{H^{k+1}} \leq e^{C|\Re\omega|} \| f\|_{H^{k}}
	\end{equation}
	for each $k \in \mathbb{N}$ and $f \in H^k(X)$.
\end{theo}

Theorem \ref{theo:maintheo} is also true when $\partial M$ consists of several Killing horizons generated by $T$, each of which has a positive, constant surface gravity. In particular, Theorem \ref{theo:maintheo} applies to any stationary perturbation of the Schwarzschild--de Sitter spacetime (which is bounded by two non-degenerate Killing horizons \cite[Section 6]{vasy2013microlocal}) that preserves the timelike nature of $T$, and for which the horizons remain non-degenerate Killing horizons. Other examples are even asymptotically hyperbolic spaces in the sense of Guillarmou \cite{guillarmou2005meromorphic}.
	

\subsection{Energy decay}
Theorem \ref{theo:maintheo} can be used to prove logarithmic decay to constants for solutions the Cauchy problem
\begin{equation} \label{eq:cauchy}
\Box_g v = 0, \quad v|_X = v_0, \quad Tv|_X = v_1.
\end{equation}
Given initial data $(v_0,v_1) \in H^{k+1}(X) \times H^k(X)$, the equation \eqref{eq:cauchy} admits a unique solution 
\[
v \in \mathcal{C}^0\left (\RR_+; H^{k+1}(X)\right) \cap \mathcal{C}^1\left(\RR_+ ;H^k(X)\right).
\]
If $N$ denotes the future pointing unit normal to the level sets of $t$ and $Q[v]$ is the stress energy tensor  (see Section \ref{subsect:stressenergy}) associated to $v$, define the energy
\[
\mathcal{E}[v](s) = \int_{\{t=s\}} Q[v](N,N)\, dS_X.
\]
Here $dS_X$ is the induced volume density on $X=\{t=0\}$, which is isometric to each time slice $\{t = s \}$. Since $N$ is timelike, it is well known that $\mathcal{E}[v](s)$ is positive definite in $dv$. One consequence of the positivity of $\kappa$ is an energy boundedness statement
\begin{equation} \label{eq:boundedenergy}
\mathcal{E}[v](t) \leq C \mathcal{E}[v](0),
\end{equation}
see for instance \cite[Corollary 3.9]{warnick2015quasinormal}. One can also define an energy $\mathcal{E}_k[v]$ controlling all derivatives up to order $k$, with $\mathcal{E}[v] = \mathcal{E}_1[v]$, which is similarly uniformly bounded. This can be improved to a logarithmic energy decay statement uniformly up to the horizon, with a loss derivatives. 
%
%
%
%
%

\begin{cor} \label{cor:cauchy}
Given $k \in \mathbb{N}$, there exists $C>0$ such that 
	\[
\mathcal{E}_k[v](t)^{1/2} \leq \frac{C}{\log(2+t)} \, \| \left( v_0, v_1 \right) \|_{\mathcal{X}^k \times H^{k+1}}
	\]
	for each $v \in \mathcal{C}^0\left (\RR_+; H^{k+1}(X)\right) \cap \mathcal{C}^1\left(\RR_+ ;H^k(X)\right)$ solving the Cauchy problem \eqref{eq:cauchy} with initial data $\left(v_0,v_1\right) \in \mathcal{X}^k \times H^{k+1}(X)$. 
\end{cor}


 We can also improve Corollary \ref{cor:cauchy} by showing that $v$ decays logarithmically to a constant as follows. 
Given $\left(v_0,v_1\right) \in \mathcal{X}^k \times H^{k+1}(X)$, define the constant
\[
v_\infty = \vol(\partial X)^{-1}\int_X \left( A^{-2} v_1 -2A^{-2}W v_0- \divop_g (A^{-2}W) v_0\right)  A \,dS_X .
\]
Here $A > 0$ is the lapse function and $W$ is the shift vector as described in Section \ref{subsect:decompose}.

\begin{cor} \label{cor:decaytoconstant}
Given $k \in \mathbb{N}$, there exists $C>0$ such that 
	\[
	\|v(t) - v_\infty \|_{H^{k+1}} + \| \partial_t v(t) \|_{H^{k}} \leq \frac{C}{\log(2+t)} \, \| \left( v_0, v_1 \right) \|_{\mathcal{X}^k \times H^{k+1}}
	\]
	for each $v \in \mathcal{C}^0\left (\RR_+; H^{k+1}(X)\right) \cap \mathcal{C}^1\left(\RR_+ ;H^k(X)\right)$ solving the Cauchy problem \eqref{eq:cauchy} with initial data $\left(v_0,v_1\right) \in \mathcal{X}^k \times H^{k+1}(X)$. 
\end{cor} 

By Sobolev embedding, Corollary \ref{cor:decaytoconstant} can be used to deduce pointwise decay estimates as well.

\subsection{Relationship with previous work} \label{subsect:previous} The analogue of Theorem \ref{theo:maintheo} was first established for compactly supported perturbations of the Euclidean Laplacian in a landmark paper of Burq \cite{burq1998decroissance}. There have been subsequent improvements and simplifications in the asymptotically Euclidean setting \cite{burq2002lower,vodev2000exponential,datchev2014quantitative}, while Rodnianski--Tao \cite{rodnianski2015effective} considered asymptotically conic spaces. In a different direction, Holzegel--Smulevici \cite{holzegel2013decay} established logarithmic energy decay on slowly rotating Kerr--AdS spacetimes, which contain a Killing horizon of the type described here in addition to a conformally timelike boundary. However, their approach made heavy use of the symmetries of Kerr--AdS, and is not adaptable to our setting.

Most relevant to the setting considered here are the works of Moschidis \cite{moschidis2016logarithmic} and Cardoso--Vodev \cite{cardoso2002uniform}. The former reference shows logarithmic energy decay on Lorentzian spacetimes which may contain Killing horizons, but importantly also contain at least one asymptotically flat end. There, the mechanism of decay is radiation into the asymptotically flat region. In contrast, asymptotically flat ends are not considered in the present paper, but we do allow spacetimes which contain Killing horizons as their only boundary components. We therefore stress that the results of \cite{moschidis2016logarithmic} are disjoint from those of this paper.

Meanwhile, \cite{cardoso2002uniform} applies to a wide class of Riemannian metrics, including those with hyperbolic ends. There is a close connection between asymptotically hyperbolic manifolds and black holes spacetimes, first exploited in the study of resonances by S\'a Barreto--Zworski  \cite{barreto1997distribution}. This relationship has attracted a great deal of interest, especially following the paper \cite{vasy2013microlocal} (for a survey of recent developments, see \cite{zworski2017mathematical}).

Common to the works described above is the use of Carleman estimates in the interior of the geometry, which is then combined with some other (typically more complicated) analysis near infinity. Although the proof of Theorem \ref{theo:maintheo} adopts techniques from \cite{burq1998decroissance,moschidis2016logarithmic,rodnianski2015effective}, one novelty (and simplifying feature) is that the Carleman estimate employed here is valid up to and including the horizon. In particular, this avoids the use of separation of variables and special function methods \cite{burq1998decroissance,holzegel2013decay,vodev2000exponential}, Mourre-type estimates \cite{burq2002lower}, and spherical energies \cite{cardoso2002uniform,datchev2014quantitative,moschidis2016logarithmic,rodnianski2015effective}.

\section{Preliminaries}

\subsection{Semiclassical rescaling}
It is  conceptually convenient to rescale the operator by
\begin{equation} \label{eq:screscaling}
P(z) = h^2 \bP(h^{-1}z).
\end{equation}
Thus $\omega = h^{-1}z$, and uniform bounds on $P(z)$ for $\pm z$ in a compact set $[a,b]\subset (0,\infty)$ give high-frequency bounds for $\bP(\omega)$ as $|\omega| \rightarrow \infty$. Theorem \ref{theo:maintheo} is easily seen to be equivalent to the following.
\begin{thmbis}{theo:maintheo} \label{theo:maintheo1}
	Given $[a,b]\subset (0,\infty)$, there exist $C,C_1 > 0$ such that
	\begin{equation} \label{eq:mainestimate1}
	\|  u \|_{H^{k+1}_h} \leq e^{C/h} \|P(z) u\|_{H^k_h} 
	\end{equation}
	for each $u\in \mathcal{X}^k$ and $\pm z \in [a,b] + ie^{-C_1/h}[-1,1]$.
\end{thmbis}
The norms in \eqref{eq:mainestimate1} are semiclassically rescaled Sobolev norms.  For detailed expositions on semiclassical analysis, the reader is referred to \cite{zworski2012semiclassical} and \cite[Appendix E]{zworski:resonances}.

\subsection{Stationarity} \label{subsect:stationarity}

A tensor on $M$ will be called stationary if it is annihilated by the Lie derivative $\mathcal{L}_T$. The definition of stationarity can be extended to $T^*M$ by observing that $T$ lifts to a vector field on $T^*M$ via the identification 
\[
T^*M = T^* \RR \oplus T^*X.
\]
Any covector $\varpi \in T^*_q M$ at a point $q =(t,x)$ can be decomposed as $\varpi = \xi + \tau  dt$, where $\xi \in T^*_x X$ and $\tau dt \in T^*_t \RR$. Thus a function $F \in \mathcal{C}^\infty(T^*M)$ is stationary if it depends only on $\xi \in T^*_x X$ and $\tau \in \RR$, which we sometimes denote by $F(x,\xi,\tau)$. Furthermore, if $\tau = \tau_0$ is fixed, then $F$ induces a function $F(\cdot,\tau_0)$ on $T^*X$.  This is compatible with the Poisson bracket in the sense that for stationary $F_1,F_2 \in \mathcal{C}^\infty(T^*M)$, there is equality
\begin{equation} \label{eq:poissoncommutes}
\{F_1,F_2\}(x,\xi,\tau_0) = \{ F_1(\cdot,\tau_0), F_2(\cdot,\tau_0)\}(x,\xi).
\end{equation}
On the left is the Poisson bracket on $T^*M$, and on the right the Poisson bracket on $T^*X$.

In particular, this discussion applies to the dual metric function $G \in \mathcal{C}^\infty(T^*M)$, whose value at $\varpi \in T^* _qM$ is given by
\[
G(x,\varpi) = g^{-1}_x(\varpi,\varpi) = g^{\alpha\beta}(x)\varpi_\alpha \varpi_\beta.
\]
The semiclassical principal symbol $p = \sigma_h(P(z))$ is given by $p(x,\xi;z) = -G(x,\xi-z\, dt)$.

\begin{lem} \label{lem:negativedefinite}
	The quadratic form $(x,\xi) \mapsto G(x,\xi)$ is negative definite on $T^*X^\circ$.
\end{lem}
\begin{proof}
The condition $\tau =0$ implies that $\varpi = \xi + 0\,dt$ is orthogonal to $T^\flat$. But $T^\flat$ is timelike on $M^\circ$, whence the result follows.	
\end{proof}

If $\tau_0 \in\RR$ is fixed and $K\subset X^\circ$ is compact, then by Lemma \ref{lem:negativedefinite} there exist $c,R>0$ such that if  $G(x,\xi) \geq R$, then 
\[
G(x,\xi + \tau_0 \, dt) \geq cG(x,\xi)
\]
for each $\xi \in T^*_K X^\circ$, where the constants $c,R$ are locally uniform in $\tau_0$.  In particular, given a compact interval $I \subset \RR$, the set 
\[
\{\xi \in T^*_K X^\circ: G(\xi + \tau\, dt) = 0 \text{ for some } \tau \in I \}
\]
is a compact subset of $T^*X^\circ$. This also implies that if $Q$ is a stationary quadratic form on $T^*M$, then there exists $C>0$ such that 
\[
|Q(x,\xi + \tau \, dt)| \leq C(1+|G(x,\xi+ \tau \, dt)|)
\]
for each $\xi \in T^*_K X^\circ$ and $\tau \in I$.

\subsection{Killing horizons and surface gravity} \label{subsect:killing}

Recall the hypotheses on $(M,g)$ described in Section \ref{subsect:mainresults}, and set 
\[
\mu = g(T,T).
\]
The key property of $(M,g)$ is that $\partial M$ is a Killing horizon generated by $T$. By definition, this means that $\partial M$ is a null hypersurface which agrees with a connected component of the set $\{\mu = 0, \,T\neq 0\}$. Of course in this case $T$ is nowhere vanishing. Since orthogonal null vectors are collinear, there is a smooth function $\kappa: \partial M \rightarrow \RR$, called the surface gravity, such that
\begin{equation} \label{eq:killingequation}
\nabla_g \mu = -2\kappa T
\end{equation}
on $\partial M$. The non-degeneracy assumption means that $\kappa > 0$, and for simplicity it is assumed that $\kappa$ is in fact constant along $\partial M$. 

\subsection{Properties of the metric} \label{subsect:decompose}
Let $N$ denote the future pointing unit normal to the level sets of $t$, and define the lapse function $A > 0$ by $A^{-2} = g^{-1}(dt,dt)$. The shift vector is given by the formula 
\[
W = T  -AN,
\] which by construction is tangent to the level sets of $t$. Let $k$ denote the induced (positive definite) metric on $X$. If $(x^i)$ are local coordinates on $X$, then
\[
g = (A^2 - k_{ij} W^i W^j) \,dt^2 - 2k_{ij} W^i  dx^j dt -k_{ij}\, dx^i dx^j. 
\]
Inverting this form of the metric gives
\begin{equation} \label{eq:inversemetric}
g^{-1} = A^{-2}(\partial_t - W^i \partial_{i})^2 - k^{ij}\partial_{i}\partial_{j}.
\end{equation}
Note that  $k(W,W) = A^2-\mu$, and hence $W\neq 0$ near $\partial M$. 

Now use the condition that $\partial M$ is a Killing horizon generated by $T$. The covariant form of \eqref{eq:killingequation} reads
\begin{equation} \label{eq:covariantkillingequation}
\partial_{i}\mu= 2\kappa W^j k_{ij}.
\end{equation}
By assumption $\kappa > 0$, so $W$ is a nonzero inward pointing normal to $X$ along $\partial X$ whose length with respect to $k$ is $A$.

Introduce geodesic normal  coordinates $(r,y^A)$ on $X$ near $\partial X$, so $r$ is the distance to $\partial X$ (uppercase indices will always range over $A = 2,\ldots, n$). By construction, $\partial_{r}$ is an inward pointing unit normal along $\partial X$, so 
\begin{equation} \label{eq:shiftvectoronhorizon}
W^r  = A, \quad W^A = 0
\end{equation}
along the boundary. Also by construction, the components of the induced metric in $(r,y^A)$ coordinates satisfy $k^{rr}=1$ and $k^{rA}=0$. 

\begin{lem}
	The function $r$ satisfies $g^{-1}(dr,dr)= -2\kappa A^{-1}r + r^2\mathcal{C}^\infty(M)$.
\end{lem}
\begin{proof}
First observe that $k_{AB}W^A W^B \in r^2 \mathcal C^\infty(M)$	by \eqref{eq:shiftvectoronhorizon}, and since $k(W,W) = A^2-\mu$,
\[
A^2- \mu =(W^r)^2 + k_{AB}W^A W^B. 
\]
Now $\mu$ and $r$ are both boundary defining functions, so $\mu= fr$ for some $f \in C^\infty(M)$, and hence $d\mu = fdr$ on $\partial X$. But on the boundary $\left<W,d\mu \right> = 2\kappa A^2$ from \eqref{eq:covariantkillingequation}, while $\left<W,dr \right> = W^r =  A$ from \eqref{eq:shiftvectoronhorizon}. Thus 
\[
\mu = fr = 2\kappa Ar + r^2\mathcal{C}^\infty(M).
\] 
 Plugging this back into the equation for $k(W,W)$ yields 
\[
(W^r)^2 = A^2 - 2\kappa A r + r^2 \mathcal C^\infty(M),
\] 
and therefore $g^{-1}(dr,dr)= -k^{rr} + A^{-2}(W^r)^2 = -2\kappa A^{-1}r + r^2\mathcal{C}^\infty(M)$ as desired. \end{proof}

Observe that the surface gravity depends on the choice of null generator $T$. Consider the rescaled vector field
\[
\widehat {T} = T/(2\kappa),
\]
which changes the time coordinate by the transformation $\widehat{t} = 2\kappa t$. If $\widehat \bP(\widehat{\omega})$ is now defined as in \eqref{eq:stationaryoperator} but replacing $t$ with $\widehat{t}$, then
\[
\bP(\omega) = \widehat{\bP}(\omega/(2\kappa)).
\]
It suffices to prove Theorem \ref{theo:maintheo} for $\widehat{\bP}(\omega)$ then, since rescaling the frequency only changes the constants $\omega_0, C_0,C$. Dropping the hat notation, it will henceforth be assumed that  $\kappa = 1/2$.

Next, consider a conformal change $g = f \tilde g$, where $f>0$ is stationary. The operator $L$ can then be written as
\begin{equation} \label{eq:conformal}
L = f^{-1}\Box_{\tilde g} + (n-1) f^{-2}\nabla_{\tilde g} f  + \mathcal{W} + \mathcal{V}.
\end{equation}
Thus we can write $L = f^{-1}\tilde L$, where $\tilde L$ has the same form as $L$ but with $\tilde g$ replacing $g$, provided that the vector field $\nabla_{\tilde g}f$ is tangent to $\partial M$. But this follows from the stationarity of $f$, since
\[
g(T,\nabla_g f) = 0
\]
and $T$ is normal to $\partial M$.
Thus it suffices to prove Theorem \ref{theo:maintheo} with $\tilde L$ replacing $L$. Observe that $\partial M$ remains a Killing horizon generated by $T$ with respect to $\tilde{g}$, and the surface gravity is unchanged.

By making a conformal change and dropping the tilde notation, it will also be assumed that
\begin{equation} \label{eq:Gdr}
g^{-1}(dr,dr) = -r.
\end{equation}
If $(\tau,\rho,\eta_A)$ are dual variables to $(t,r,y^A)$, define a stationary quadratic form $G_0 \in \mathcal{C}^\infty(T^*M)$ by
\begin{equation} \label{eq:G_0}
G_0 = -r\rho^2 - 2\rho \tau - k^{AB}_0\eta_A \eta_B.
\end{equation}
Here $k_0$ is the restriction of $k$ to $\partial M$, which is then extended to a neighborhood of $\partial M$ by requiring that $\mathcal{L}_{\partial_r} k_0 = 0$. In the next section, the difference $G-G_0$ will be analyzed.

\subsection{Negligible tensors} \label{subsect:negligible} In this section we define a class of tensors which will arise as errors throughout the proof of Theorem \ref{theo:maintheo1}.

\begin{defi}
	\begin{inparaenum}[1)]
		\item A stationary $1$-tensor $F^\alpha \partial_{\alpha}$ is said to be negligible if its components in a coordinate system $(t,r,y^A)$ satisfy
	\[
	F^t \in r\mathcal{C}^\infty(M), \quad F^r \in r^2\mathcal{C}^\infty(M), \quad F^A \in r\mathcal{C}^\infty(M).
	\]
	\item  A stationary $2$-tensor $H^{\alpha\beta} \partial_{\alpha}\partial_{\beta}$ is said to be negligible if its components in a coordinate system $(t,r,y^A)$ satisfy
	\begin{align*}
	\begin{cases}
	H^{tt} \in \mathcal{C}^\infty(M), \quad H^{rr} \in r^2 \mathcal{C}^\infty(M), \quad H^{AB} \in  r \mathcal{C}^\infty(M) , \\H^{tA} \in \mathcal{C}^\infty(M), \quad H^{tr} \in r\mathcal{C}^\infty(M), \quad H^{rA} \in  r \mathcal{C}^\infty(M).
	\end{cases}
	\end{align*}
	\end{inparaenum}
	Observe that negligibility is invariant under those coordinate changes which leave $(t,r)$ invariant. Denote by $\Neg_1$ and $\Neg_2$ all $\mathcal{C}^\infty(T^*M)$ functions of the form $F^\alpha \varpi_\alpha$ and  $H^{\alpha\beta}\varpi_\alpha\varpi_\beta$, respectively.
	\end{defi} 

Recall the definition of $G_0$ in \eqref{eq:G_0}. The notion of negligibility is motivated by the fact that
\[
G = G_0 + \Neg_2.
\]
This follows directly from \eqref{eq:inversemetric}, \eqref{eq:shiftvectoronhorizon}, and \eqref{eq:Gdr}. We will also repeatedly reference the auxiliary functions 
\begin{equation} \label{eq:YZ}
Y = (r\rho)^2 + \tau^2, \quad Z = r\rho^2 + k^{AB}\eta_A\eta_B.
\end{equation}
It follows immediately from the Cauchy--Schwarz inequality $2ab < \delta a^2 + b^2/\delta$ that there exists $C>0$ satisfying
\begin{equation} \label{eq:ZG0bound}
Z \leq C\left( |G_0| + \tau^2/r \right).
\end{equation}
 The next two lemmas also follow from judicious applications of the Cauchy--Schwarz inequality and the trivial observation that $(r\rho)^2 = r(r\rho^2)$ is small relative to $r\rho^2$ for small values of $r$.

\begin{lem} \label{lem:neg1bounds}
	Let $F \in \Neg_1$. Then, for each $\gamma > 0$ there exists $C_\gamma$ such that
	\[ 
	r^{-1} |\tau| |F| \leq C_\gamma \tau^2 + \gamma Z.
	\]
Furthermore, $\rho \Neg_1 \subset \Neg_2$ and $\Neg_1 \cdot \Neg_1 \subset r \Neg_2$.
	\end{lem}

\begin{lem} \label{lem:neg2bounds}
	Let $H\in \Neg_2$. Then, for each $\gamma > 0$ there exist $C_\gamma, r_\gamma>0$ such that
	\[ 
	|H| \leq C_\gamma  Y + \gamma k^{AB}\eta_A \eta_B, \quad 
	|H| \leq C_\gamma \tau^2 + \gamma Z
	\]
	for $r \in [0,r_\gamma]$. 
\end{lem}

Now combine Lemma \ref{lem:neg2bounds} with the bound \eqref{eq:ZG0bound} and the relation $G = G_0 + \Neg_2$. Thus there exists $R > 0$ and $C>0$ such that
\begin{equation} \label{eq:ZGbound}
Z \leq C(|G| + \tau^2/r)
\end{equation}
for $r \in [0,R]$.

The next goal is to compute the Poisson brackets $\{G,r\}$ and $\{G,\{G,r\}\}$. To begin, observe that 
\begin{equation} \label{eq:G0poisson}
\{G_0,r\} = -2(r\rho + \tau), \quad \{G_0,\{G_0,r\}\} = 2(r\rho^2 + 2\tau \rho).
\end{equation}
In order to replace $G_0$ with $G$ we also need to consider the Poisson brackets of functions in $\Neg_1$ and $\Neg_2$.

\begin{lem} \label{lem:poissonbracket}
	The Poisson bracket satisfies $\{\mathcal{N}_2,r \} \subset \mathcal{N}_1$ and $\{\Neg_2,\Neg_1\} \subset \Neg_2$, as well as 
	$\{ G_0,\Neg_1\} \subset \Neg_2$ and $\{\{G_0,r\},\Neg_2\} \subset \Neg_2$. Therefore,
	\begin{equation} \label{eq:Gpoisson}
	\{G,r\} = -2(r{\rho}+\tau) + \Neg_1, \quad \{G,\{G,r\}\} = 2(r\rho^2 + 2\tau \rho) + \Neg_2.
	\end{equation}
	Furthermore, $\{G,\{G,r\}\} = -2r\rho^2 + \Neg_2$ whenever $\{G,r\}=0$. 
\end{lem}
\begin{proof}
	The first part is a direct calculation, while \eqref{eq:Gpoisson} follows from the first part and \eqref{eq:G0poisson}. The last statement follows from the inclusion $\rho \Neg_1 \subset \Neg_2$.
\end{proof}

\section{Carleman estimates in the interior} \label{sect:interiorcarleman}

\subsection{Statement of result} In this section we prove a Carleman estimate valid in the interior $X^\circ$, but with uniform control over the exponential weight near $\partial X$. 

Recall that $r$ denotes the distance on $X$ to the boundary with respect to the induced metric. Although this function is only well defined in a small neighborhood of $\partial X$, for notational convenience we will assume that $[0,3]$ is contained in the range of $r$ (otherwise it is just a matter of replacing $3$ with $3\varepsilon$ for an appropriate $\varepsilon  > 0$).

\begin{prop} \label{prop:interiorcarleman}
	Given $[a,b]\subset (0,\infty)$, there exists $r_1 \in (0,1)$ and $\varphi_1,\varphi_2 \in \mathcal{C}^\infty(X)$ such that
	\begin{itemize} \itemsep6pt
		\item on $\{r\leq 1\}$ the functions $\varphi_1,\varphi_2$ are equal and depend only on $r$,
		\item $\varphi_i'(r) < 0$ is constant on $\{r \leq r_1\}$ for $i=1,2$,
		\end{itemize}
	with the following property: given a compact set $K \subset X^\circ$ there exists $C>0$ such that
	\[
\| (e^{\varphi_1/h} + e^{\varphi_2/h})u \|_{H^2_h(X)} \leq  C	h^{-1/2} \| (e^{\varphi_1/h} + e^{\varphi_2/h})P(z)u \|_{L^2(X)}
	\]
	for each $u \in \mathcal{C}_c^\infty(K^\circ)$ and $\pm z \in [a,b]$.	 
\end{prop}

It clearly suffices to prove Proposition \ref{prop:interiorcarleman} for the operator $L = \Box_g$, since the lower order terms can be absorbed as errors. In order to prove Theorem \ref{theo:maintheo1}, an additional estimate is needed near the boundary; this is achieved in Section \ref{sect:boundarycarleman} below.

\subsection{The conjugated operator} \label{subsect:conjugated} Given $\varphi \in \mathcal{C}^\infty(X)$, define the conjugated operator
\[
P_{\varphi}(z) = e^{\varphi/h}P(z) e^{-\varphi/h}.
\]
Let $p_\varphi(z)$ denote its semiclassical principal symbol. Define $L^2(X)$ with respect to the density $A\cdot dS_X$, where recall $dS_X$ is the induced volume density on $X$, and $A>0$ is the lapse function as in Section \ref{subsect:killing}. Defining $\Re P_{\varphi}(z)$ and $\Im P_{\varphi}(z)$ with respect to this inner product, integrate by parts to find 
\begin{align} \label{eq:carleman}
\| P_\varphi(\omega) u \|^2_{L^2(X)} =  \left<P_\varphi(\omega)P_\varphi(\omega)^*u,u\right>_{L^2(X)} +  i\left<[\Re P_\varphi(\omega), \Im P_\varphi(\omega)] u, u \right>_{L^2(X)}
\end{align}
for $u \in \mathcal{C}^\infty_c(X^\circ)$. The idea is to find $\varphi$ which satisfies H\"ormander's hypoellipticity condition
\begin{equation} \label{eq:hormander}
\{\Re p_\varphi, \Im p_\varphi\} > 0
\end{equation}
on the characteristic set $\{p_\varphi = 0\}$.

In order to apply the results of Section \ref{subsect:negligible} without introducing additional notation, it is convenient to work with the dual metric function $G$ directly. Define
\[
G_\varphi(x,\varpi) = G(x,\varpi + id\varphi),
\]
so since we are assuming that $\tau$ is real, $\Re G_{\varphi}(x,\varpi) = G(x,\varpi) - G(x,d\varphi)$, and $\Im G_{\varphi}(x,\varpi) = (H_G \varphi)(x,\varpi)$. We will then construct $\varphi$ (viewed as a stationary function on $M$) such that
\begin{equation} \label{eq:Gbracket}
\{\Re G_{\varphi}, \Im G_{\varphi} \}(x,\varpi) = \left(H^2_G\varphi\right)(x,\varpi)  + \left(H^2_G \varphi\right) (x,d\varphi) > 0
\end{equation}
on $\{G_{\varphi} =0\} \cap \{ a \leq \pm \tau \leq b \}$. This will imply the original hypoellipticity condition from the discussion surrounding \eqref{eq:poissoncommutes} and the identifications 
\[
p_\varphi(x,\xi;z) = -G_\varphi(x,\xi - z\, dt), \quad z = -\tau.
\]
Note the the dual variable $\tau$ is now playing the role of a \emph{rescaled} time frequency.

\subsection{Constructing the phase in a compact set} To avoid any undue topological restrictions, we will actually construct two weights $\varphi_1, \varphi_2$ in the interior, which agree outside a large compact set. This appears already in \cite{burq1998decroissance}, but we will follow the closely related presentation in \cite{moschidis2016logarithmic,rodnianski2015effective}.

\begin{lem} \label{lem:psi}
	There exist positive functions $\psi_1,\psi_2 \in \mathcal{C}^\infty(X)$ with the following properties.
	\begin{enumerate} \itemsep6pt
		\item $\psi_1,\psi_2$ have finitely many non-degenerate critical points, all of which are contained in $\{ r > 2\}$.
		\item $\psi_2 > \psi_1$ on $\{d\psi_1 = 0\}$, and $\psi_1 > \psi_2$ on $\{d\psi_2 = 0\}$.
		\item The functions $\psi_1, \psi_2$ are equal and depend only on $r$ in $\{r \leq 2 \}$. Furthermore $\partial_r\psi_1$ and $\partial_r\psi_2$ are negative in this region.
	\end{enumerate}
\end{lem}

\begin{proof}
Let $\zeta  \in \mathcal{C}^\infty(\{ r \geq 2\})$ solve the boundary value problem
\[
\Delta_k \zeta = 1, \quad \zeta |_{\{r=2\}} = 1.
\]
Here $\Delta_k$ is the non-positive Laplacian with respect to the induced metric $k$. Since $\Delta_k \zeta> 0$, none of the critical points of $\zeta$ in $\{r>2\}$ are local maxima. In addition, since $\zeta$ clearly achieves its maximum at each point of $\{r=2\}$, its outward pointing normal derivative is strictly positive by Hopf's lemma \cite[Lemma 3.4]{gilbarg2015elliptic}. By construction, the outward pointing unit normal is $-\partial_r$, hence $\zeta' < 0$ near $\{r=2\}$ (for the remainder of the proof, prime will denote differentiation with respect to $r$).

The first step is to replace $\zeta$ by a Morse function. We may for instance embed $\{r\geq 2\}$ into a compact manifold $X_0$ without boundary, and approximate an arbitrary smooth extension of $\zeta$ to $X_0$ by a Morse function in the $\mathcal{C}^\infty(X_0)$ topology. Restricting to $\{r\geq 2\}$ and again calling this replacement $\zeta$, we still have that $\zeta$ has no local maximum in $\{r > 2\}$ and $\zeta' < 0$ near $\{r =2\}$. In particular, all critical points of $\zeta$ are nondegenerate and lie in a compact subset of $\{r > 2\}$.

%

Now fix any function $\bar \zeta = \bar \zeta(r) \in \mathcal{C}^\infty(\{r < 3\})$ such that $\bar \zeta' < 0$ everywhere,  and $\bar \zeta \geq \zeta$ on their common domain of definition $\{ 2 \leq r < 3\}$. Choose a cutoff $H = H(r) \in \mathcal{C}^\infty(X;[0,1])$ such that 
\[
H=1 \text{ for } r < 2+ \gamma , \quad \supp H \subset \{r \leq 2+2\gamma\},
\]
and $H' \leq 0$. Set $\psi_1 = H \bar \zeta + (1-H) \zeta$, and compute $\psi_1' = H'(\bar \zeta - \zeta) + H \bar \zeta' + (1-H) \zeta'$. If $\gamma > 0$ is sufficiently small, then $\psi'_1 < 0$ in a neighborhood of $\supp H$, since the sum of the last two terms is strictly positive on $\supp H$.  On the other hand, outside of such a neighborhood the only critical points of $\psi_1$ are those of $\zeta$. 

Let $p_1,\ldots,p_n$ enumerate the necessarily finite number of critical points of $\psi_1$, and choose $\gamma>0$ such that the closed geodesic balls $B(p_1,\gamma), \ldots, B(p_n,\gamma)$ are mutually disjoint and $B(p_j,\gamma) \subset \{r > 2\}$ for each $j$. Since $p_j$ is not a local maximum, for each $j$ there is a point $q_j \in B(p_j,r)$ such that 
\[
\psi_1(q_j) > \psi_1(p_j).
\] 
Now choose a diffeomorphism $g : X \rightarrow X$ which is the identity outside the union of the $B(q_j,r)$ and exchanges $p_j$ with $q_j$. Then, set $\psi_2 = \psi_1 \circ g$. By construction the only critical points of $\psi_2$ are $q_1,\ldots,q_n$, and furthermore
\[
\psi_2(p_j) > \psi_1(p_j), \quad \psi_1(q_j) > \psi_2(q_j)
\]
for each $j$. Since outside of $\{ r>2\}$ the functions $\psi_1 = \psi_2$ depend on $r$ only, the proof is complete, adding an appropriate constant if necessary to ensure that both functions are positive.
\end{proof}

Let $B_1 \subset \{r > 2\}$ be a closed neighborhood of $\{d\psi_1=0\}$ such that $\psi_2 > \psi_1$ on $B_1$, and likewise for $B_2$, exchanging the roles of $\psi_1$ and $\psi_2$. Also, let $U_i \subset B_i$ be additional neighborhoods of $\{d\psi_i =0\}$. Now define 
\begin{equation} \label{eq:exponentiated}
\varphi_i = \exp(\alpha \psi_i), \quad i=1,2,
\end{equation}
where $\alpha > 0$ is a parameter. The following lemma is a standard computation which is included for the sake of completeness.
 \begin{lem} \label{lem:interiorphase}
	Given $\varepsilon>0$ and $\tau_0 > 0$, there exists $\alpha_0 > 0$ such that if $\alpha \geq \alpha_0$, then
	\[
	\{\Re G_{\varphi_i}, \Im G_{\varphi_i}\} > 0
	\]
	on $\left(\{G_{\varphi_i}=0\} \cap \{r \geq \varepsilon\} \cap \{ |\tau| \leq \tau_0 \}\right) \setminus T^*_{U_i} M$ for $i=1,2$.
\end{lem}
\begin{proof}
The subscript $i=1,2$ will be suppressed. Use the definition \eqref{eq:exponentiated} to compute
\[
H_G\varphi = \alpha e^{\alpha\psi}H_G \psi, \quad H^2_G \varphi = \alpha^2 e^{\alpha \psi}(H_G\psi)^2 + \alpha e^{\alpha \psi}H^2_G\psi.
\]
Assume that $G_{\varphi}(x,\varpi) = 0$. It follows from $\Im G_\varphi(x,\varpi) =0$ that $(H_G \varphi)(x,\varpi) = 0$, and hence $(H_G \psi)(x,\varpi)=0$. Therefore by \eqref{eq:Gbracket}, 
\begin{multline*}
\{G - G(x,d\varphi), H_G \varphi \}(x,\varpi) \\= \alpha e^{\alpha \psi}(H^2_G\psi)(x,\varpi)  + \alpha^3 e^{3\alpha \psi}(H^2_G \psi)(x,d\psi) +  \alpha^4 e^{3\alpha \psi}|G(x,d\psi)|^2.
\end{multline*}
Next, use the condition $(\Re G_\varphi)(x,\varpi) = 0$, which implies that $G(x,\varpi)=\alpha^2 e^{2\alpha\psi}G(x,d\psi)$. By the discussion following Lemma \ref{lem:negativedefinite}, there exists $C>0$ such that 
\[
|(H^2_G \psi)(x,\varpi)| \leq  C (1+ |G(x,\varpi)|)
\]  
on $\{r \geq \varepsilon\}\cap\{|\tau| \leq \tau_0\}$. Thus on the set $\{ G_\varphi = 0\} \cap \{r \geq \varepsilon\}\cap\{|\tau| \leq \tau_0\}$, 
\[
|\alpha e^{\alpha \psi} (H^2_G\psi)(x,\varpi)|+  |\alpha^3 e^{3\alpha \psi}(H^2_G \psi)(x,d\psi)| \leq C \alpha^3 e^{3\alpha \psi}.
\]
On the other hand, as soon as $d\psi \neq 0$ the third term $\alpha^4 e^{3\alpha \psi}|G(x,d\psi)|^2$ is positive by Lemma \ref{lem:negativedefinite}, and dominates the previous two terms for large $\alpha > 0$. Since $d\psi \neq 0$ away from $B$, the proof is complete. \end{proof}

\subsection{Constructing the phase outside of a compact set}
 The most delicate part of the argument is the construction of the phase outside of a compact set. 
 Since $g^{-1}(dr,dr) = -r$ and $\varphi$ is a function only of $r$ in this region,
\[
G_\varphi = G + r(\varphi')^2 + i\varphi' H_G r. 
\] 
Now compute the Poisson bracket
\begin{align*}
\{\Re G_\varphi, \Im G_\varphi\} &= \{G+r(\varphi')^2, \varphi' H_G r \} \\ &= \varphi' H^2_G r + \varphi'' (H_G r)^2 - \left((\varphi')^3 + 2r(\varphi')^2 \varphi''\right) \partial_{\rho} H_G r.
\end{align*}
Assume that $\varphi' < 0$, in which case $\Im G_\varphi = 0$ is equivalent to $H_G r = 0$. The goal is then to arrange negativity of the term
\begin{equation} \label{eq:boundaryhypo}
H^2_G r - \left(  (\varphi')^2 + 2r\varphi' \varphi'' \right)\partial_{\rho} H_G r
\end{equation}
on the set $\{\Re G_\varphi = 0\}$. Recall the definition of $Z$ from \eqref{eq:YZ}.

\begin{lem} \label{lem:Zboundchar}
	There exists $C >0$ and $R>0$ such that
	$Z \leq C(r(\phi')^2 + \tau^2/r)$ on $\{\Re G_\varphi=0\}\cap \{0 < r \leq R\}$.
\end{lem}
\begin{proof}
Apply \eqref{eq:ZGbound}, using that $\Re G_\varphi=0$ implies $G = -r(\varphi')^2$.
\end{proof}

Putting everything together, it is now easy compute $H^2_G r$ on $\{G_\varphi = 0\}$ near the boundary.

\begin{lem} \label{lem:H2pcalculation}
	For each $\delta >0$ there exists $R_\delta > 0$ such that
	\[
	|H^2_G r +2\tau^2/r| \leq \delta(r(\phi')^2 + \tau^2/r)
	\] 
	on $\{G_\varphi=0\} \cap \{0< r\leq R_\delta\}$.
\end{lem}
\begin{proof}
From the expression \eqref{eq:Gpoisson} for $H^2_G r$ and Lemma \ref{lem:neg2bounds}, find $C_\gamma>0$ and $r_\gamma>0$ such that 
	\begin{equation} \label{eq:bound3}
	|H^2_G r +2r\rho^2| < C_\gamma |\tau|^2 + \gamma Z
	\end{equation}
	for $r \in (0, r_\gamma)$. Now multiply $H_G r$ by $\rho$, and use that $\rho \Neg_1 \subset \Neg_2$. Therefore by Lemma \ref{lem:neg2bounds}, there exists $C'_\gamma >0$ and $r'_\gamma > 0$ such that
	\begin{equation} \label{eq:bound2}
	|2r\rho^2 + 2\tau\rho| < C'_\gamma |\tau|^2 + \gamma Z
	\end{equation}
	for $r \in (0,r'_\gamma)$. On the other hand, from $H_G r = 0$, deduce that $-\tau\rho = \tau^2/r + \tau r^{-1}\Neg_1$. By Lemma \ref{lem:neg1bounds}, there exists $C''_\gamma>0$ such that
	\begin{equation} \label{eq:bound1}
	|2\tau\rho+ 2\tau^2/r| < C''_\gamma|\tau|^2  +\gamma Z.
	\end{equation}
Combine \eqref{eq:bound3}, \eqref{eq:bound2}, and \eqref{eq:bound1} via the triangle inequality with Lemma \ref{lem:Zboundchar} to find that
	\begin{align*}
	|H^2_G r + 2 \tau^2/r| < 3\gamma  C (r(\phi')^2 + \tau^2/r)   +\left(C_\gamma + C'_\gamma + C_\gamma''\right)\tau^2 
	\end{align*}
	for $r \in (0, \min\{r_\gamma,r'_\gamma,R\})$; here $C>0$ and $R>0$ are provided by Lemma \ref{lem:Zboundchar}. Finally, choose $\gamma$ sufficiently small depending on $\delta$ and a corresponding $R_\delta > 0$ such that the conclusion of the lemma holds for $r \in (0,R_\delta)$.
\end{proof}

Next, observe that $-\partial_{\rho}H_G r = 2r + r^2\mathcal{C}^\infty(M)$. Given $a>0$, it follows from \eqref{eq:boundaryhypo} and Lemma \ref{lem:H2pcalculation} that there exists $R_1 > 0$ such that 
\begin{equation} \label{eq:quantitativebracket}
(\varphi')^{-1}\{\Re G_\varphi, \Im G_\varphi\} < -3a^2/(2r) +  3 r(\varphi')^2 + 3 r^2\varphi' \varphi''
\end{equation}
on $\{G_\varphi=0\} \cap \{0< r\leq R_1\} \cap \{|\tau| \geq a\}$, provided that $\varphi'' \geq 0$. 

Shrinking $R_1$ if necessary, it may be assumed that $\psi = \psi_i$ as in Lemma \ref{lem:psi}  satisfies $\psi' < 0$ on $[0,R_1+1]$. Recalling that $\varphi_i = \exp(\alpha \psi_i)$, choose $\alpha>0$ satisfying the conclusion of Lemma \ref{lem:interiorphase} with $\varepsilon = R_1$.
By further increasing $\alpha$, (but keeping $a>0$ fixed), it may also be assumed that $\varphi = \varphi_i$ satisfies
 \begin{equation} \label{eq:alphainequalities}
 3(\varphi'(R_1)R_1)^2 > a^2, \quad \varphi''(r) \geq -\varphi'(r)/r \text{ for } r\in [R_1,R_1+1].
 \end{equation}
 Although $\varphi$ is already defined on all of $X$, the following lemma allows one to redefine $\varphi$ on $\{r < R_1 + 1\}$ in such a way that its derivative is controlled; this new extension will still be denoted by $\varphi$. The idea comes from \cite[Section 3.1.2]{burq1998decroissance}, but of course the form of the operator there is quite different.

\begin{lem}
	There exists an extension of $\varphi =\varphi_i$ from $\{r \geq R_1+1\}$ to $\{r < R_1+1\}$ such that 
	\[
\{\Re G_\varphi, \Im G_\varphi\} > 0
	\]
	on $\{G_\varphi=0\} \cap \{0< r\leq R_1\} \cap \{|\tau| \geq a\}$. Furthermore, there exists $r_1 \in (0, R_1)$ such that $\varphi'(r) < 0$ is constant for $r\in  [0,r_1]$.
\end{lem}
\begin{proof} 
Motivated by \eqref{eq:quantitativebracket}, consider the differential equation
\[
-a^2/r + 3rk^2 + 3r^2 k k' = 0, \quad k(R_1) = \varphi'(R_1)< 0.
\]
This is a Bernoulli equation whose solution is given by 
\[
k(r) = -r^{-1}\left(\left(\varphi'(R_1)R_1 \right)^2 + (2/3)a^2\log (r/R_1) \right)^{1/2}.
\]
The solution is certainly meaningful for $r \in [R_0,R_1]$, where we define $R_0$ by
\[
R_0 = R_1\exp\left( 1/2-(3/2)\left(\varphi'(R_1)R_1/a\right)^2\right).
\] 
Note that we indeed have $R_0 < R_1$ by the assumption \eqref{eq:alphainequalities}. The value $R_0$ was chosen such that $k'(R_0) =0$, and it is easy to see that $k'(r) > 0$ for $r \in (R_0,R_1]$. In addition, $k(R_0) < 0$. Let $\theta = \theta(r)$ be defined on $[0,R_1+1]$ by
\[
\theta(r) = \begin{cases} \varphi'(r), \quad & r\in [R_1,R_1+1], \\
k(r), \quad & r\in [R_0,R_1], \\
k(R_0), \quad & r \in [0,R_0].  \end{cases}
\] 
The function $\theta$ is strictly negative, and the piecewise continuous function $\theta'$ satisfies $-a^2/r + 3r\theta^2 + 3r^2 \theta \theta'  \leq 0$
for $r\in(0,R_1+1]$. Indeed, by construction of $k$ and $R_0$, the inequality holds for $r\in (0,R_1)$, and it is also true for $r \in (R_1,R_1+1]$ by \eqref{eq:alphainequalities}. Rearranging,
\begin{equation} \label{eq:thetainequality}
\theta' \geq a^2/(3r^3\theta) - \theta/r
\end{equation}
for $r \in (0,R_1+1]$.

 We now proceed to mollify $\theta$ in such a way that the hypotheses of the lemma hold. Let $\eta_\varepsilon(r) = (1/\varepsilon)\eta(r/\varepsilon)$ denote a standard mollifier, where $\eta \in \mathcal{C}^\infty_c((-1,1))$ has integral one. In addition, choose a cutoff $H = H(r) \in \mathcal{C}^\infty(X;[0,1])$ such that
\[
H=1 \text{ for } r < R_1 + 1/4 , \quad H=0 \text{ for } r > R_1 + 1/2,
\]
and $H' \leq 0$. Now define
\[
\theta_\varepsilon = (1-H)\theta + \eta_\varepsilon *(H\theta).
\]  
Clearly $\theta_\varepsilon$ is smooth, and $\theta_\varepsilon\rightarrow \theta$ uniformly for $r\in [0,R_1+1]$. Furthermore, there exists $\varepsilon_0 > 0$ such that if $\varepsilon \in (0,\varepsilon_0)$, then the following properties are satisfied:

\begin{itemize} \itemsep6pt 
		\item $\theta_\varepsilon(r) < 0$ and $\theta_\varepsilon'(r) \geq 0$ for $r\in [0,R_1+1]$.
	
	\item $\theta_\varepsilon(r) = \varphi'(r)$ for $r \in [R_1+3/4,R_1+1]$,

	\item There exists $r_1 \in (0,R_0]$ such that $\theta_\varepsilon(r) = k(R_0)$ for $r\in [0,r_1]$.

\end{itemize}
Since $\theta$ is continuous and piecewise smooth,
\begin{equation} \label{eq:thetaprime}
\theta_\varepsilon' = (1-H)\theta' - H'\theta + \eta_\varepsilon * (H'\theta + H \theta').
\end{equation}
Therefore by \eqref{eq:thetainequality},
\begin{align*}
\theta_\varepsilon' &\geq - H'\theta + \eta_\varepsilon * \left (H'\theta \right) \\&+  (1-H)\left(a^2/(3r^3\theta) - \theta/r\right) + \eta_\varepsilon * \left( H \left(a^2/(3r^3\theta) - \theta/r\right) \right)
\end{align*}
for $r\in (0,R_1+1]$.
The right-hand side converges uniformly to $a^2/(3r^3\theta) - \theta/r$ for $r\in [r_1,R_1+1]$ since the latter function is continuous there. Since $\theta_\varepsilon\rightarrow \theta$ uniformly for $r\in [r_1,R_1 + 1]$ as well, there exists $\varepsilon \in (0,\varepsilon_0)$ such that
\[
-3a^2/(2r) +  3 r \theta_\varepsilon^2 + 3r^2\theta_\varepsilon \theta_\varepsilon' \leq 0
\]
for $r\in [r_1,R_1+1]$. This inequality is also true for $r\in (0,r_1)$, since $\theta_\varepsilon = k(R_0)$ on that interval. Now extend $\varphi$ from $\{r \geq R_1+1\}$ to $\{ r< R_1 +1 \}$ by the formula
\[
\varphi(r) = \varphi(R_1+1) + \int_{R_1+1}^r \theta_\varepsilon(s)\, ds.
\]
This completes the proof according to \eqref{eq:quantitativebracket} by observing that the $\varphi$ just constructed satisfies $\varphi''(r) \geq 0$.
\end{proof}

As a remark, if $\tau \neq 0$, then the hypoellipticity condition also holds along $\{r=0\}$, simply because $\Im G_\varphi \neq 0$ in that case. However, since $(x,\xi) \mapsto G(x,\xi)$ is not elliptic along $\{r=0\}$, the hypoellipticity condition alone, stated here in the semiclassical setting, is not sufficient to prove a Carleman estimate --- cf. \cite[Section 8.4]{hormander2013linear}

Now that the phases $\varphi_1,\varphi_2$ have been constructed globally, we are ready to finish the proof of Proposition \ref{prop:interiorcarleman}. Here we come back to the operator $P_\varphi(z)$ on $X$. Fix a norm $|\cdot|$ on the fibers of $T^*X$ (for instance using the induced metric $k$) and let $\left<\xi \right> = (1+|\xi|^2)^{1/2}$.

\begin{proof}[Proof of Proposition \ref{prop:interiorcarleman}]

Recall that we are given $[a,b]\subset (0,\infty)$ and a compact set $K\subset X^\circ$. Without loss, we may assume that $K = \{r \geq \varepsilon\}$ for some $\varepsilon>0$. Let $B_i, U_i$ be  as in the discussion preceding Lemma \ref{lem:interiorphase}. In particular, 
\[
\{\Re p_{\varphi_i}, \Im p_{\varphi_i}\} > 0
\]
on $\left(\{p_{\varphi_i}=0\} \cap \{r \geq \varepsilon/2\} \right) \setminus T^*_{U_i} X$. Let $\chi_i \in \mathcal C_c^\infty(B_i^\circ)$ be such that $\chi_i =1$ near $U_i$. If $\varphi = \varphi_i$, then
 \[
 |p_\varphi|^2 +\chi^2 + h\{\Re p_\varphi, \Im p_\varphi\}  \geq h\left( M |p_\varphi|^2 + M\chi^2 + \{\Re p_\varphi, \Im p_\varphi\} \right)
 \]
for any $M>0$, provided that $h>0$ is sufficiently small. On the other hand, the set $\{ \Re p_\varphi = 0\}\cap \{r \geq \varepsilon/2\}$ is compact by Lemma \ref{lem:negativedefinite}, uniformly for $\pm z \in [a,b]$. Therefore,
\[
\left<\xi\right>^{-4} \left( M |p_\varphi|^2 + M\chi^2 + \{\Re p_\varphi, \Im p_\varphi\} \right) > 0
\]
near $T^*X \cap \{r\geq \varepsilon/2\}$ for $M>0$ sufficiently large. By \eqref{eq:carleman} and the semiclassical G\aa rding inequality applied to $e^{\varphi_i/h} u$,
\begin{equation} \label{eq:varphii}
h\| e^{\varphi_i/h}u\|^2_{H^2_h(X)} \leq C\| e^{\varphi_i/h}P(z)u \|^2_{L^2(X)} + C\| e^{\varphi_i/h} u\|^2_{L^2(B_i)}
\end{equation}
for $u \in \mathcal C_c^\infty(K^\circ)$ and $i=1,2$. Since $\varphi_1 > \varphi_2$ on $B_2$ and $\varphi_2 > \varphi_1$ on $B_1$, there is $\gamma > 0$ such that
\[
e^{\varphi_i/h} \leq e^{-\gamma/h}\left( e^{\varphi_1/h} + e^{\varphi_2/h} \right)
\]
on $B_i$. Now add \eqref{eq:varphii} for $i=1,2$ to absorb the integral over $B_1 \cup B_2$ into the left-hand side.
\end{proof}

 \section{Degenerate Carleman estimates near the boundary} \label{sect:boundarycarleman}

\subsection{Statement of result}

In this section we complement Proposition \ref{prop:interiorcarleman} with a result valid up to the boundary. Recall that the phases $\varphi_1,\varphi_2$ are equal on $\{ r \leq 1\}$. Since we are working near $\partial X$, we will thus drop the subscript and simply write $\varphi$.

\begin{prop} \label{prop:boundarycarleman} 
	Given $[a,b]\subset (0,\infty)$ there exists $r_0>0$ and $C>0$ such that 
	\begin{equation} \label{eq:boundarycarleman}
	\| e^{\varphi/h} u \|_{H^1_{b,h}} \leq C\left( h^{-1/2}\| e^{\varphi/h}P(z)u \|_{L^2} + e^{\varphi(0)/h}  \|u\|_{L^2(\partial X)} \right).
	\end{equation}
	for $u \in \mathcal{C}_c^\infty(\{ r < r_0\})$ and $\pm z \in [a,b]$.
\end{prop}

The Sobolev space appearing on the left-hand side of \eqref{eq:boundarycarleman} is modeled on the space of vector fields $\mathcal{V}_b(X)$ which are tangent to the boundary; see \cite{melrose1993atiyah}. Thus $u \in H^1_{b}(X)$ if $u\in L^2(X)$ and $Ku \in L^2(X)$ for any $K \in \mathcal{V}_b(X)$. If $u \in H^1_{b}(X)$ and $\supp u \subset \{r < 1\}$, we can set
\[
\| u \|^2_{H^1_{b,h}} = \int_{X} |u|^2 + h^2|r\partial_r u|^2 + h^2k^{AB} \left(\partial_{A} u\cdot \partial_{B} \bar u \right) \, dS_X .
\]
Of course away from $\partial X$ this is equivalent to the full $H^1_h$ norm. Observe that it is enough to prove Proposition \ref{prop:boundarycarleman} for the operator $L= \Box_g$, since the estimate \eqref{eq:boundarycarleman} is stable under perturbations $B \in h\mathrm{Diff}^1_h(X)$ provided that the vector field part of $B$ is tangent to $\partial X$. The latter condition is satisfied by the hypothesis that $\mathcal{W}$ is tangent to $\partial M$ made in the introduction.

Proposition \ref{prop:boundarycarleman} is proved through integration by parts. A convenient way of carrying out this procedure is by constructing an appropriate multiplier for the wave operator and applying the divergence theorem. This approach to Carleman estimates for certain geometric operators is partly inspired by \cite{alexakis2015global,ionescu2009uniqueness}.

\subsection{The divergence theorem}
We will use the divergence theorem in the time-differentiated form
\begin{equation} \label{eq:divergence}
\frac{d}{dt} \int_{X} g(K,N) \, dS_X + \int_{\partial X} g(K,T) \, dS_{\partial X} = \int_{X} \left(\divop_g K \right)  A\,dS_X,
\end{equation}
valid for any vector field $K$ (see \cite[Lemma 3.1]{warnick2015quasinormal} for instance), where recall $X = \{t=0\}$. Thus the first term on the left-hand side of \eqref{eq:divergence} is short-hand for
\[
\frac{d}{ds} \int_{\{t=s\}} g(K,N) \, dS_X \text{ evaluated at } s=0. 
\]
Here $dS_{\partial X}$ is the volume density on $\partial X$ induced by $k$ (the latter is Riemannian, hence the induced volume density is well defined).

\subsection{Stress-energy tensor} \label{subsect:stressenergy}
Given $v \in \mathcal{C}^\infty(M)$, let $ Q = Q[v]$ denote the usual stress energy tensor associated to $v$ with components
\[
Q_{\alpha\beta} = \Re \left( \partial_{\alpha} v\cdot  \partial_{\beta} \bar v\right) - (1/2)g^{-1}(dv,d\bar{v}) g_{\alpha\beta}.
\]
This tensor has the property that $(\nabla^\beta Q_{\alpha\beta}) S^\alpha = \Re (\Box v \cdot S\bar v)$ for any vector field $S$. Given such a vector field and a function $w$, define the modified vector field $J= J[v]$ with components
\[
J^\alpha = Q^{\alpha}_{\beta}S^\beta + (1/2)w \cdot \partial^{\alpha} (|v|^2) - (1/2)(\partial^\alpha  w)|v|^2.
\]
The relevant choices in this context are 
\begin{equation} \label{eq:Sw}
S = \nabla_g r, \quad w = \lambda + (1/2)\Box_g r,
\end{equation}
where $\lambda = \lambda(r)$ is an undetermined function to be chosen in Lemma \ref{lem:pseudoconvex} below. Also, introduce the tensor $\Pi$ with components
\[
\Pi^{\alpha\beta} = -\nabla^{\alpha\beta}r - \lambda g^{\,\alpha \beta}.
\]
The divergence of $J$ satisfies 
\begin{equation} \label{eq:Jdivergence}
\Re \left(\Box_g u \cdot (S\bar v+w\bar v)\right) = \mathrm{div}_g J + \Pi(dv,d\bar v) + (1/2)(\Box_g w)|v|^2,
\end{equation}
which is verified by a direct calculation.

\subsection{The conjugated operator}
Near $\partial M$, consider the conjugated operator $L_\Phi = e^{\Phi}\Box_g e^{-\Phi}$, where $\Phi=\Phi(r)$. Then, $L_\Phi$ has the expression
\begin{align*}
L_\Phi &= \Box_g - 2\Phi' S+ ( (\Phi')^2-\Phi'')g^{-1}(dr,dr) - \Phi' \Box_g r
\\& = \Box_g - 2\Phi' S + V_0.
\end{align*}
Now $g^{-1}(dr,dr)= -r$ by assumption, and consequently the potential term $V_0$ satisfies 
\[
{V}_0 = r (\Phi''-(\Phi')^2) - \Phi'\Box_g r.
\] 
Set ${V}_1 =  V_0 - 2\Phi'w$, multiply $L_\Phi v$ by $S\bar v + w\bar v$, and take the real part to find that
\begin{align} \label{eq:Lidentity}
\Re (L_\Phi v\cdot (S\bar v + w\bar v)) &= \Re (\Box_g v \cdot(S\bar v + w\bar v)) - 2\Phi'|Sv|^2 \notag \\  &+ \Re V_1 v\cdot S\bar v + {V}_0 w|v|^2.
\end{align}
It is also convenient to write $\Re \left( {V}_1 v\cdot S\bar v\right)$ as a divergence,
\[
\Re \left({V}_1 v \cdot  S\bar v\right) = (1/2)\divop_g \left(V_1 |v|^2 S\right)- (1/2)\left( S({V}_1) +{V}_1 \Box_g r\right)|v|^2.
\]
In view of this expression, define the vector field $K= J + (1/2) V_1 |v|^2 S$.
For future use, also define the modified potential ${V}$ by
\begin{equation} \label{eq:modifiedpotential}
{V} = (1/2)(\Box_g w)  + {V}_0 w -(1/2)S({V}_1) - (1/2) {V}_1 \Box_g r + \Phi' w^2.
\end{equation}
On one hand, integrating the divergence of $K$ yields boundary integrals; the following special case of this will suffice.
\begin{lem} \label{lem:boundary}
	Let $v \in \mathcal{C}^\infty(M)$ be given by $v=e^{-izt/h}u$, where $u$ is stationary and $z \in \RR$. Then,
	\[
	\int_X \left(\mathrm{div}_g K\right) A \, dS_X = -|z/h|^2 \int_{\partial X} |u|^2 \, dS_{\partial X}.
	\] 
\end{lem}
\begin{proof}
	Apply the divergence theorem \eqref{eq:divergence}. Since $z \in \RR$, the vector field $K$ is stationary, and hence there is no contribution from the time derivative. As for the integral over $\partial M$, observe that $T$ is null and $S=-T$ on the horizon. Since $Tv = -i(z/h) v$, it follows that $g(T,K) = -|Tv|^2 = -|z/h|^2 |u|^2$ on $\partial M$.
\end{proof}

Note that the boundary contribution from Lemma \ref{lem:boundary} has an unfavorable sign, which will account for the boundary term in Proposition \ref{prop:boundarycarleman}. On the other hand, the divergence of $K$ can also be expressed in terms of \eqref{eq:Lidentity}.

\begin{lem} \label{lem:fundamental}
If $\Phi'<0$, then the divergence of $K$ satisfies
\begin{equation} \label{eq:fundamentalinequality}
(2|\Phi'|)^{-1}|L_\Phi v|^2 \geq  \mathrm{div}_g K+ \Pi(dv,d\bar v) - \Phi' |Sv|^2 + {V}|v|^2,
\end{equation}
where $V$ is given by \eqref{eq:modifiedpotential}.
\end{lem}
\begin{proof}
Combine \eqref{eq:Lidentity} with \eqref{eq:Jdivergence}, and then use the Cauchy--Schwarz inequality to find
\[
\Re \left(L_\Phi v \cdot (S\bar v + w\bar v)\right) \leq (2|\Phi'|)^{-1} |L_\Phi v|^2 - \Phi' \left(|Sv|^2 + w^2|v|^2 \right),
\]
recalling that $\Phi' < 0$.
\end{proof}

\subsection{Pseudoconvexity}

To examine positivity properties of $\Pi(dv,d\bar v) - \Phi'|Sv|^2$, we establish a certain pseudoconvexity condition. A criterion of this type first appeared in work of Alinhac on unique continuation \cite{alinhac1984unicite}, and was also employed in \cite{ionescu2009uniqueness,alexakis2015global}. Recall that the Poisson bracket is related to the Hessian via the formula
\begin{equation} \label{eq:covariant}
 \{G,\{G,f\}\}(x,\varpi)= 4 \varpi_{\alpha}\varpi_\beta \nabla^{\alpha\beta} f,
\end{equation}
valid for any $f \in \mathcal{C}^\infty(M)$. 

\begin{lem} \label{lem:pseudoconvex}
	There exists $M,c,R_0>0$, and a function $\lambda = \lambda(r)$ such that 
	\begin{equation} \label{eq:pseudoconvex}
	M\{G,r\}^2 - \{G,\{G,r\}\} - 4\lambda G \geq c\left((r\rho)^2 +\tau^2+ k^{AB}\eta_A \eta_B\right) 
	\end{equation}
	for $r\in [0,R_0]$.
\end{lem} 	
\begin{proof}
	Throughout, assume that $M\geq 1$. Let $r \leq (4M)^{-1}$, and define the function $\lambda$ by
	\[
	\lambda = (1/2) -(1-\delta)r M,
	\]
	where $\delta >0$ will be chosen sufficiently small. Observe that $1/4 \leq \lambda \leq 1/2$ uniformly in $M \geq 1$ for $r \leq (4M)^{-1}$. Denote the left-hand side of \eqref{eq:pseudoconvex} by $4\mathcal{E}$, and the corresponding quantity by $4\mathcal{E}_0$ if $G$ is replaced with $G_0$. Dividing through by four,
	\begin{equation} \label{eq:pseudoconvexG0}
	\mathcal{E}_0 = M((r\rho)^2 + 2 r\rho \tau + \tau^2) - (1/2)(r\rho^2 + 2\rho \tau) - \lambda G_0.
	\end{equation}
	Use the expression for $\lambda G_0$ and the lower bound $\lambda \geq 1/4$ on $\{ r \leq (4M)^{-1}\}$ to find that
	\[
	\mathcal{E}_0 \geq M\delta((r\rho)^2 + 2 r\rho \tau) + M\tau^2 + (1/4)k_0^{AB}\eta_A \eta_B.
	\]
	Therefore $\mathcal{E}_0 \geq c(M Y + k^{AB}\eta_A \eta_B)$ if $\delta>0$ is sufficiently small, where recall $Y = (r\rho)^2 +\tau^2$.

	Now consider the error $\mathcal{E}-\mathcal{E}_0$ incurred by replacing $G$ with $G_0$. Replacing $M\{G,r\}^2$ with $M\{G_0,r\}^2$ produces an error 
	\[
	2M\{G_0,r\}\{G-G_0,r\} + M\{G-G_0,r\}^2.
	\] 	
	Using Cauchy--Schwarz on the first term to absorb a small multiple of $M\{G_0,r\}^2$ into $\mathcal{E}_0$ (in other words, changing the constant $c>0$ in the lower bound for $\mathcal{E}_0$ above) leaves an overall error of the form 
	\[
	M \left( \Neg_1\cdot \Neg_1\right) \subset \left( r M \right) \Neg_2.
	\] 
	The factor of $r M$ is harmless since $r M \leq 1/4$, thus the right-hand side is certainly in $\Neg_2$ uniformly in $M \geq 1$. Using that $\lambda$ is uniformly bounded in $M \geq 1$ on $\{r \leq (4M)^{-1}\}$, the remaining errors $\lambda (G-G_0)$ and
	\[
	\{G-G_0,\{G-G_0,r\}\} + \{G-G_0,\{G_0,r\}\} + \{G_0,\{G-G_0,r\}\} 
	\] 
	are also in $\Neg_2$ by Lemma \ref{lem:poissonbracket}, uniformly in $M \geq 1$. Now apply the first bound in Lemma \ref{lem:neg2bounds}, choosing $\gamma > 0$ sufficiently small but independent of $M$ so that $\gamma k^{AB}\eta_A\eta_B$ can be absorbed by $c k^{AB}\eta_A \eta_B$ on the right-hand side for $ r\in [0,r_\gamma]$. This leaves a large multiple of $Y$, which is then absorbed by $M Y$ on the right-hand side by taking $M$ sufficiently large. It then suffices to take $R_0 = \min\{(4M)^{-1}, r_\gamma\}$.
	\end{proof}

Fix $M>0$ such that Lemma \ref{lem:pseudoconvex} is valid. This fixes the function $\lambda$, and therefore the function $w$ in \eqref{eq:Sw}.  Lemma \ref{lem:fundamental} will be applied with the weight $\Phi = \varphi_i/h$, viewed as a stationary function on $M$. In particular, $\Phi' = -C/h$ on $\{ r \leq r_1 \}$ for some constant $C>0$ (recall the statement of Proposition \ref{prop:interiorcarleman}).

Before proceeding, consider the potential term ${V}$ from Lemma \ref{lem:fundamental}. Instead of analyzing its sign, we more simply note that for $F' = -C/h$ one has
\begin{equation} \label{eq:potential}
{V} = f_0 + h^{-1} f_1 + h^{-2}f_2, 
\end{equation}
where $f_0, f_1 \in \mathcal{C}^\infty(M)$ and  $f_2 \in r\mathcal{C}^\infty(M)$. The small coefficient of $f_2$ means ${V}$ can be treated as an error. To be precise, we have the following positivity result for the bulk terms.

\begin{lem} \label{lem:positivebulk}
Given $a>0$, there exists $c, r_0 \geq 0$ such that if $|z| \geq a$, then
\begin{equation} \label{eq:positivebulk}
\Pi(dv,d\bar v) - \Phi'|Sv|^2 + {V}|v|^2 \geq c\left(h^{-2}|u|^2 + |r\partial_r u|^2 + k^{AB}  \partial_{A} u \, \partial_{B} \bar u \right)
\end{equation}
on $\{r \leq r_0\}$ for each $v\in\mathcal{C}^\infty(M)$ of the form $v = e^{-iz t/h}u$, where $u$ is stationary.
\end{lem}
\begin{proof}
	Since $\Phi' = -C/h$, an inequality of the form \eqref{eq:positivebulk} is true for sufficiently small $h>0$ if the term ${V}|v|^2$ is dropped from the left-hand side; this follows from Lemma \ref{lem:pseudoconvex} and \eqref{eq:covariant}. On the other hand, for a potential ${V}$ satisfying \eqref{eq:potential}, there is clearly $r_0 > 0$ such that $V|v|^2$ can be absorbed by $ch^{-2}|v|^2$ for $r \in [0,r_0]$ and $h>0$ sufficiently small.
\end{proof}

The proof of Proposition \ref{prop:boundarycarleman} is now immediate:

\begin{proof} [Proof of Proposition \ref{prop:boundarycarleman}]
Given $[a,b]\subset (0,\infty)$, apply Lemmas \ref{lem:boundary}, \ref{lem:fundamental},  \ref{lem:positivebulk} to functions of the form $v = e^{-i z t/h} e^{\varphi/h} u$, where $\pm z \in [a,b]$ and $\supp u \subset \{r < r_0\}$.  
\end{proof}

\section{Proof of Theorem \ref{theo:maintheo}} \label{sect:proof}

We prove the equivalent Theorem \ref{theo:maintheo1}. Assume that $[a,b]\subset (0,\infty)$ has been fixed. Choose a cutoff function $\chi \in \mathcal{C}^\infty(X)$ such that 
\[
\supp \chi \subset \{r < r_0\}, \quad \chi = 1 \text{ near }\{ r \leq r_0/2\},
\] 
where $r_0$ is provided by Lemma \ref{lem:positivebulk}. Then, apply Proposition \ref{prop:boundarycarleman} to $\chi u$ and Proposition \ref{prop:interiorcarleman} to $(1-\chi)u$, where $u \in \mathcal{C}^\infty(X)$. Since the commutator $[P(z),\chi]$ is supported away from $\partial X$, the error terms can be absorbed even though the left-hand side is only estimated in the $H^1_{b,h}$ norm. Bounding $e^{\varphi_1/h} + e^{\varphi_2/h}$ from below on the left and from above on the right yields
\begin{equation} \label{eq:degenerateestimate}
\|u \|_{H^1_{b,h}} \leq e^{C/h} \left( \| P(z)u \|_{L^2} + \|u\|_{L^2(\partial X)} \right)
\end{equation}
for $u \in \mathcal{C}^\infty(X)$ and $\pm z \in [a,b]$.

 Next, we remove the boundary term on the right-hand side of \eqref{eq:degenerateestimate}. In order to estimate the boundary term, we use that $L$ is formally self-adjoint and that $\mathcal{W}$ is tangent to $\partial M$. Apply the divergence theorem \eqref{eq:divergence} to the vector field $\bar v \nabla_g v -  v \nabla_g \bar v + |v|^2\cdot  \mathcal{W}$  with $v = e^{-i zt/h}u$. Since $L$ is formally self-adjoint, we obtain Green's formula
\[
(hz) \int_{\partial X} |u|^2 \, dS_{\partial X} = -\Im \int_{X} P(z)u \cdot \bar u \, A \,dS_X.
\]
There is no boundary contribution coming from $\mathcal{W}$ since we assumed $g(T,\mathcal{W})$ vanishes on $\partial M$. Applying Cauchy--Schwarz to the right-hand side implies that
\[
e^{C/h}\| u \|_{L^2(\partial X)} \leq C_\varepsilon h^{-1} e^{2C/h}\|P(z) u\|_{L^2} + \varepsilon\| u \|_{L^2}
\]
for some $C_\varepsilon$ and every $\pm z \in [a,b]$. Therefore the boundary term on the right-hand side of \eqref{eq:degenerateestimate} can be absorbed into the left-hand side by taking $\varepsilon$ sufficiently small, at the expense of increasing the constant in the exponent $e^{C/h}$. We then have
\[
\|u \|_{H^1_{b,h}} \leq e^{C/h} \| P(z)u \|_{L^2}. 
\]
The final step is to apply a bound of the form
\begin{equation} \label{eq:absorbing}
\| u \|_{H^{k+1}_h} \leq Ch^{-1} \big( \| P(z)u \|_{H^{k}_h} + \| u \|_{L^2}\big)
\end{equation}
for $u \in \mathcal{C}^\infty(X)$ and $\pm z \in [a,b]$. The most conceptual way of understanding this estimate is in terms of the semiclassical trapping present in the interior of $M$. For an appropriate pseudodifferential complex absorbing operator $Q \in \Psi^{-\infty}_h(X^\circ)$ with compact support in $X^\circ$, the nontrapping framework of \cite[Section 2.8]{vasy2013microlocal} shows that $P(z)-iQ$ satisfies the nontrapping bounds
\[
\| u \|_{H^k_h} \leq Ch^{-1} \| (P(z) - iQ)u \|_{H^k_h}
\]
for $z \in [a,b]$. Here $Q$ is chosen to be elliptic (with the correct choice of sign) on the trapped set. In this case $Q$ can be chosen to have compact microsupport in $X^\circ$, hence maps $Q : \mathcal C^{-\infty}(X) \rightarrow \mathcal C^\infty(X)$, and in particular
\[
\| Qu \|_{H^k_h} \leq C \| u\|_{L^2}.
\]
This clearly implies \eqref{eq:absorbing} for $z \in [a,b]$, with a similar argument when $-z \in [a,b]$.

%
%

This completes the proof of Theorem \ref{theo:maintheo1} in the case when $u \in \mathcal{C}^\infty(X)$ and $\pm z \in [a,b]$. By perturbation, this extends to a region $\pm z \in [a,b] +ie^{-C_1/h}[-1,1]$. Simply write 
\[
P(z) - P(\Re z) = \Im z \cdot B(z),
\] 
where $B(z) \in \mathrm{Diff}^1_h(X)$ is bounded $H^{k+1}_h(X) \rightarrow H^k_h(X)$ uniformly for $z \in [a,b]$ (although $B(z)$ is not holomorphic in $z$). Thus the difference can be absorbed into the left-hand side if $|\Im z| \leq e^{-C_1/h}$ for $C_1>0$ sufficiently large. Finally, $\mathcal{C}^\infty(X)$ is dense in $\mathcal{X}^k$ (cf. \cite[Lemma E.47]{zworski:resonances}), so \eqref{eq:mainestimate1} is valid for $u \in \mathcal{X}^k$ as well, thus completing the proof of Theorem \ref{theo:maintheo1}.

\section{Logarithmic energy decay}

\subsection{A semigroup formulation} \label{subsect:semigroup}
In this section we outline how Corollary \ref{cor:cauchy}  can be deduced from the resolvent estimate \eqref{eq:mainestimate} via semigroup theory. The starting point is that the Cauchy problem \eqref{eq:cauchy} is  associated with a $\mathcal{C}^0$ semigroup $U(t) = e^{-itB}$ on  $\mathcal{H}^k = H^{k+1}(X)\times H^{k}(X)$
satisfying 
\begin{equation} \label{eq:growth}
\| U(t) \|_{\mathcal{H}^k\rightarrow \mathcal{H}^k} \leq Ce^{\nu t}
\end{equation}
for some $C,\nu >0$ \cite[Corollary 3.14]{warnick2015quasinormal}. Recalling the lapse function $A = g^{-1}(dt,dt)^{-1/2}$, write
\[
\Box_g = L_2 + L_1 \partial_t + A^{-2}\partial_t^2,
\]
where $L_j$ is identified with a differential operator on $X$ of order $j$. Thus $L_2 = \bP(0)$ and $L_1 = i\partial_\omega \bP(0)$. More explicitly,
\[
L_1 = -2A^{-2}W - \divop_g (A^{-2}W),
\] 
where $W$ is the shift vector from Section \ref{subsect:decompose}. The infinitesimal generator is then given by
\begin{equation} \label{eq:B}
-iB = \begin{pmatrix}
0 & 1 \\ -A^{2}L_2 & -A^{2}L_1
\end{pmatrix}.
\end{equation}
Indeed, applying $U(t)$ to initial data in $\mathcal{C}_c^\infty(X^\circ)$ shows that  $-iB$ is given by \eqref{eq:B} in the sense of distributions. Now the resolvent set of $B$ is non-empty, and indeed $\sigma(B) \subset \{\Im \omega \leq \nu \}$ by \eqref{eq:growth}. Therefore the domain $D(B)$ of $B$ is characterized as those distributions $(v_0,v_1)  \in \mathcal{H}^k$ such that
\[
v_1 \in H^{k+1}(X), \quad L_2 v_0 + L_1 v_1 \in H^k(X).
\]
Since $L_2 = \bP(0)$ and $L_1 \in \mathrm{Diff}^1(X)$, this shows that the domain of $B$ is 
\[
D(B) =  {\mathcal{X}}^k \times H^{k+1}(X),
\]
where $\mathcal{X}^k$ is defined by \eqref{eq:Xspace}. It is also easy to see that the graph norm on $D(B)$ satisfies
\[
\| B(v_0,v_1) \|_{\mathcal{H}^k} + \| (v_0,v_1)\|_{\mathcal{H}^k} \leq C \| (v_0,v_1) \|_{\mathcal{X}^k\times H^{k+1}},
\]
hence the two norms on $D(B)$ are equivalent by the open mapping theorem. Furthermore, the spectrum of $B$ in $\{ \Im \omega > -\kappa(k+1/2)\}$ coincides with poles of $\bP(\omega)^{-1}$, and the resolvent estimate \eqref{eq:mainestimate} translates into the bound $\| (B-\omega)^{-1} \|_{\mathcal{H}^k\rightarrow \mathcal{H}^k} \leq e^{C|\Re \omega|}$ for $\omega \in \Omega$.

\subsection{Logarithmic stabilization of semigroups}
The goal now is to apply a theorem on the logarithmic stabilization of certain bounded semigroups:
\begin{theo}[{\cite[Theorem 3]{burq1998decroissance}, \cite[Theorem 1.5]{batty2008non}}] \label{theo:semigroupdecay} Let $U(t) = e^{-itB}$ be a bounded $\mathcal{C}^0$ semigroup on a Hilbert space $\mathcal{H}$. If $\sigma(B) \cap \RR = \emptyset$ and $\| (B-\omega)^{-1}\|_{\mathcal{H}\rightarrow \mathcal{H}} \leq e^{C|\omega|}$ for $\omega \in \RR$, then there exists $C>0$ such that
	\[
	\| U(t)v \|_{\mathcal{H}} \leq \frac{C}{\log(2+t)} \| (B-i)v\|_{\mathcal{H}}
	\]
	for each $v \in D(B)$.
\end{theo}

A priori the semigroup $U(t)$ from Section \ref{subsect:semigroup} is not uniformly bounded in time on $\mathcal{H}^k$, since the energy $\mathcal{E}_k[v](t)$ does not control the $L^2$ norm of $v(t)$. Instead, observe that $\mathrm{span}\{(1,0)\} \subset \mathcal{H}^k$ is invariant under $U(t)$, which therefore descends to a semigroup $\widehat{U}(t)$ on the quotient space 
\[
\widehat{\mathcal{H}}^k = \mathcal{H}^k/\mathrm{span}\{(1,0)\}.
\] 
If $\pi:\mathcal{H}^k\rightarrow \widehat{\mathcal{H}}^k$ is the natural projection, then, the infinitesimal generator of $\widehat{U}(t)$ is simply the operator $\widehat{B}$ induced by $B$ on $\pi(D(B))$. It follows from \eqref{eq:boundedenergy} and the Poincar\'e inequality that $\widehat{U}(t)$ is a bounded $\mathcal{C}^0$ semigroup.

Since $\mathrm{span}\{(0,1)\}$ is finite-dimensional, the spectrum of $\widehat{B}$ is contained in the spectrum of $B$, and furthermore the bound 
\[
\| (\widehat B-\omega)^{-1} \|_{\widehat{\mathcal{H}}\rightarrow \widehat{\mathcal{H}}} \leq e^{C|\Re \omega|}
\]
also holds for $\omega \in \Omega$. The final step is to show $\sigma(\widehat{B}) \cap \RR = \emptyset$. If $\omega \in \RR \setminus 0$, this follows from the fact that $\bP(\omega)^{-1}$ has no nonzero real poles \cite[Lemma A.1]{warnick2015quasinormal}. 

Finally, consider the spectrum at $\omega = 0$. If $\omega_0$ is a pole of  $(B-\omega)^{-1}$ acting on $\mathcal{H}^k$ with $\Im \omega_0 > -\kappa(k+1/2)$, then its Laurent coefficients all map into $\mathcal{C}^\infty(X) \times \mathcal{C}^\infty(X)$ \cite[Section 2.6]{vasy2013microlocal}. Thus $\ker B \subset \mathcal{X}^k \times H^{k+1}(X)$ is in one-to-one correspondence with smooth stationary solutions of $\Box_g v = 0$. If $\Box_g v=0$ for $v$ smooth and stationary, then \eqref{eq:divergence} applied to the vector field $\bar v \nabla_g v + v \nabla_g \bar v$
shows that $g^{-1}(dv,d\bar v) =0$ on $X$. Again using that $v$ is stationary, Lemma \ref{lem:negativedefinite} implies that $dv=0$, and hence $v$ is constant. Thus $\ker B = \mathrm{span}\{(1,0)\}$, so $0 \notin \sigma(\widehat{B})$.

The hypotheses of Theorem \ref{theo:semigroupdecay} are therefore satisfied by $\widehat{U}(t)$, which yields the bound
\begin{equation} \label{eq:quotientdecay}
\| \widehat{U}(t)\circ\pi(v_0,v_1) \|_{\widehat{\mathcal{H}}^k} \leq \frac{C}{\log(2+t)} \| (B-i)(v_0,v_1) \|_{{\mathcal{H}^k}} 
\end{equation}
for each $(v_0,v_1)\in \mathcal{X}^k\times H^{k+1}(X)$. This establishes Corollary \ref{cor:cauchy}, since the norm on the left-hand side of \eqref{eq:quotientdecay} is equivalent to $\mathcal{E}_k[v](t)^{1/2}$, where $v$ solves the Cauchy problem \eqref{eq:cauchy} with initial data $(v_0,v_1)$.

\subsection{Decay to a constant}
To prove Corollary \ref{cor:decaytoconstant}, consider the Laurent expansion of $(B-\omega)^{-1}$ about $\omega = 0$. The range of the corresponding residue $\Pi_0$ consists of all generalized eigenvectors, and contains $\mathrm{span}\{(1,0)\}$. 

If the algebraic multiplicity of $\omega =0$ was greater than one, then there would exist a solution of $\Box_g v=0$ of the form
\[
v(t,x) = u(x) + t,
\]
where $u \in \mathcal{C}^\infty(M)$ is stationary.
This is compatible with energy boundedness, but not with the logarithmic energy decay established above. Thus $\omega=0$ is a simple pole with algebraic multiplicity one. 

By standard spectral theory, $\Pi_0$ is the projection onto $\mathrm{span}\{(1,0)\}$ along $
\mathrm{range}(B)$, so
\[
\Pi_0 = \left<\cdot,\psi\right>(1,0)
\]
for some $\psi \in (\ker B)'$, which we identify with $(\mathcal{H}^k)'/\mathrm{range}(B^*) = \ker(B^*)$. Furthermore, $\psi$ is uniquely determined by requiring that $\left<(1,0),\psi\right> =1$. Here the duality between $\mathcal{H}^k$ and 
\[
(\mathcal{H}^k)' = \dot{H}^{-k-1}(X)\times \dot{H}^{-k}(X)
\] 
is induced by the $L^2(X)$ inner product described in Section \ref{subsect:conjugated}, where $\dot{H}^s(X)$ is the Sobolev space of supported distributions in the sense of \cite[Appendix B.2]{hormander1994analysis}. 

The domain of $B^*$ consists of all $w \in \dot{H}^{-k-1}(X)\times \dot{H}^{-k}(X)$ for which there exists $v \in \dot H^{-k-1}(X)\times \dot H^{-k}(X)$ satisfying $(w,Bu) = (v,u)$ for every $u \in D(B) = \mathcal{X}^k\times H^{k+1}(X)$. Thus 
\[
D(B^*) = \dot H^{-k-1}(X) \times \dot{\mathcal{X}}^{-k},
\]
where we define 
\[
\dot{\mathcal{X}}^{-k} = \{u\in \dot{H}^{-k}(X): \bP(0) \in \dot{H}^{-k-1}(X)\}.
\]
The action of $B^*$ is given by
\[
iB^* =
\begin{pmatrix}
0& -L_2A^2 \\1 & L_1A^2
\end{pmatrix},
\]
using that $L_1$ is skew-adjoint. 

Now we compute the kernel of $B^*$, which again by abstract spectral theory is one-dimensional. Let $\psi_1 = \vol(\partial X)^{-1}A^{-2} \in L^2(X)$, viewed as an element of $\dot{H}^{-k}(X)$ via the $L^2(X)$ inner product, and then set
\[
\psi_0 = -\vol (\partial X)^{-1} L_1(1) \in \dot{H}^{-k-1}(X)
\]
in the sense of supported distributions. If we set $\psi = (\psi_0,\psi_1)$, then $B^*\psi =0$. Furthermore,
\begin{align*}
\vol (\partial X)\left<1,\psi_0\right> = \left<L_1(1),1\right> &= - \int_X \divop_g (A^{-2}W) \,A\, dS_X \\ &= -\int_{\partial X} A^{-2} g(W,T) \, dS_{\partial X} = \int_{\partial X} \, dS_{\partial X} = \vol (\partial X),
\end{align*}
since $g(W,T) = -g(AN,T) = -A^2$ on $\partial X$. Thus $\psi \in \ker B^*$ has the appropriate normalization.

Finally, let $E = \mathrm{range}(I-\Pi_0)$, which is thus invariant under $U(t)$, and $U(t)|_{E} = U(t)(I-\Pi_0)$. Since
\[
\mathcal{H}^k = E \,\dot{+} \,\mathrm{span}\{(1,0)\}
\] 
with $\dot{+}$ denoting a topological direct sum, 
it follows that $E$ is isomorphic to the quotient $\widehat{\mathcal{H}}^k$ as a Banach space. Given $(v_0,v_1) \in D(B)$, define the constant $v_\infty = \left<v_0,\psi_0\right> + \left<v_1,\psi_1 \right>$. Then
\[
\| U(t)(v_0-v_\infty,v_1) \|_{\mathcal{H}^k} = \| U(t)(v_0,v_1) - (v_\infty,0) \|_{\mathcal{H}^k} \leq C\| \widehat{U}(t)\circ\pi(v_0,v_1) \|_{\widehat{\mathcal{H}}^k},
\]
which completes the proof of Corollary \ref{cor:decaytoconstant}.

\section*{Acknowledgments}
This work was supported by the NSF grant DMS-1502632. The author would like to thank Georgios Moschidis for carefully explaining parts of his paper \cite{moschidis2016logarithmic}, and Daniel Tataru for a discussion on Carleman estimates.

\bibliographystyle{alphanum} 
	
	\bibliography{bib}

\end{document}